\documentclass[12pt,svgnames]{article}
\usepackage{amsmath}
\usepackage{amssymb}
\usepackage{amsthm}
\usepackage{wasysym}
\usepackage{enumerate}
\usepackage{dsfont}
\usepackage{stmaryrd}
\usepackage{comment}
\usepackage{tikz}
\usepackage{tikz-cd}
\usepackage{extarrows}
\usepackage{mathrsfs}
\usepackage{enumitem}
\usepackage{fullpage}
\usepackage{hyperref}
\usepackage[capitalize,nameinlink]{cleveref}
\usepackage{titlesec}
\usepackage[all]{xy}
\hypersetup{
    linkbordercolor = {PaleTurquoise},
    citebordercolor = {PaleGreen},
}

\titleformat{\section}{\normalsize\bfseries}{\thesection}{1em}{}
\titleformat{\subsection}{\normalsize\bfseries}{\thesubsection}{1em}{}

\numberwithin{equation}{subsection}

\newtheorem{defn}{Definition}[section]
\newtheorem{para}[defn]{}
\newtheorem{prop}[defn]{Proposition}
\newtheorem{theo}[defn]{Theorem}
\newtheorem{lem}[defn]{Lemma}
\newtheorem{cor}[defn]{Corollary}
\newtheorem{rmk}[defn]{Remark}

\newtheorem{egg}[defn]{Example}

\newtheorem{condition}[defn]{Condition}
\newtheorem{conjecture}[defn]{Conjecture}
\newtheorem*{claim*}{Claim}
\newtheorem*{just*}{Justification}
\newtheorem*{assumption*}{Assumption}
\newtheorem*{lem*}{Lemma}
\newtheorem*{prop*}{Proposition}
\newtheorem*{thm*}{Theorem}

\newcommand{\V}{\mathscr{V}}

\newcommand{\tensor}{\otimes}

\newcommand{\Cat}{\text{-}\mathsf{Cat}}

\newcommand{\underJ}{\kern -0.5ex \mathscr{J}}
\newcommand{\Set}{\mathsf{Set}}

\newcommand{\T}{\mathbb{T}}
\newcommand{\C}{\mathscr{C}}

\newcommand{\B}{\mathscr{B}}
\newcommand{\A}{\mathscr{A}}

\newcommand{\Mod}{\text{-}\mathsf{Mod}}

\newcommand{\Xbar}{\overline{X}}

\newcommand{\E}{\mathsf{E}}

\newcommand{\calE}{\mathcal{E}}

\newcommand{\Var}{\mathsf{Var}}

\newcommand{\Str}{\mathsf{Str}}

\newcommand{\R}{\mathcal{R}}

\newcommand{\PMet}{\mathsf{PMet}}
\newcommand{\Met}{\mathsf{Met}}

\newcommand{\sfk}{\mathsf{k}}

\begin{document}

\title{\Large \textbf{Extensivity of categories of relational structures}}
\author{Jason Parker
\medskip \\
\small Brandon University, Brandon, Manitoba, Canada}
\date{}

\maketitle

\begin{abstract}
We prove that the category of models of any relational Horn theory satisfying a mild syntactic condition is infinitely extensive. Central examples of such categories include the categories of preordered sets and partially ordered sets, and the categories of small $\V$-categories, (symmetric) pseudo-$\V$-metric spaces, and (symmetric) $\V$-metric spaces for a commutative unital quantale $\V$. We also explicitly characterize initial sources and final sinks in such categories, and in particular embeddings and quotients. 
\end{abstract}

\section{Introduction}

A category $\C$ is \emph{infinitely extensive} \cite{Carboniextensive} if it has small coproducts and for any small family $(X_i)_{i \in I}$ of objects of $\C$, the canonical functor $\prod_{i} \C/X_i \to \C/\left(\coprod_i X_i\right)$ is an equivalence. Prominent and well-known examples of infinitely extensive categories include cocomplete elementary toposes (such as Grothendieck toposes), the category $\mathsf{Cat}$ of small categories and functors, and the category $\mathsf{Top}$ of topological spaces and continuous maps. Generalizing the last example, Mahmoudi-Schubert-Tholen showed in \cite{Universalitycoproducts} that many of the categories studied in \emph{monoidal topology} \cite{Monoidaltop} (which are defined in terms of commutative unital quantales $\V$) are also infinitely extensive. Clementino \cite{Clementinoextensive} recently extended the results of \cite{Universalitycoproducts} to a more general setting (where a commutative unital quantale $\V$ is replaced by a complete and cocomplete symmetric monoidal closed category $\V$). 

As a further contribution to this line of work, we show in this article that if $\T$ is a \emph{relational Horn theory} satisfying a mild syntactic condition \eqref{variable_para}, then the category $\T\Mod$ of $\T$-models and their morphisms is infinitely extensive. As we show in \cref{without_equality_examples}, key examples of such categories include: the category $\mathsf{Preord}$ of preordered sets and monotone maps, and its full subcategory $\mathsf{Pos}$ of partially ordered sets; for a commutative unital quantale $\V$, the category $\V\Cat$ of (small) $\V$-categories and $\V$-functors, the category $\PMet_\V$ of \emph{(symmetric) pseudo-$\V$-metric spaces} and \emph{$\V$-contractions}, and its full subcategory $\Met_\V$ of \emph{(symmetric) $\V$-metric spaces}. The infinite extensivity of some of these categories (e.g.~$\mathsf{Preord}$ and $\V\Cat$) is already known from the results of \cite{Universalitycoproducts, Clementinoextensive}; but the categories studied in \cite{Universalitycoproducts, Clementinoextensive} do not subsume all of the categories studied in the present article.\footnote{For instance, the categories studied in \cite{Universalitycoproducts, Clementinoextensive} are all \emph{topological} over $\Set$, while not all categories studied in the present article (e.g.~$ \mathsf{Pos}$) have this property. Moreover, it does not seem that the categories $\PMet_\V$ and $\Met_\V$ of \emph{symmetric} (pseudo-)$\V$-metric spaces are examples that are (directly) captured by \cite{Universalitycoproducts, Clementinoextensive}.}  On the other hand, the categories considered herein are all \emph{locally presentable} \cite{LPAC}, and there are certainly examples of non-locally-presentable categories studied in \cite{Universalitycoproducts, Clementinoextensive} (e.g.~$\mathsf{Top}$) whose infinite extensivity is therefore not a consequence of the results of the present article. In summary, the results of the present article neither subsume nor are subsumed by the results of \cite{Universalitycoproducts, Clementinoextensive}.

We now provide an outline of the article. After recalling some relevant background on concrete and topological categories in \S\ref{background}, we begin \S\ref{first_section} by defining the notion of a \emph{relational signature} $\Pi$, which is a set of \emph{relation symbols} equipped with an assignment to each relation symbol of a finite positive arity. We then define the concrete category $\Str(\Pi)$ of \emph{$\Pi$-structures} and \emph{$\Pi$-morphisms}. For a regular cardinal $\lambda$, we define the notion of a \emph{$\lambda$-ary relational Horn theory} $\T$ over a relational signature $\Pi$, and we provide some central examples of such theories. In \S\ref{second_section} we study some topological properties of the category $\T\Mod$ for a relational Horn theory $\T$ \emph{without equality}, and we provide an explicit characterization of initial sources (and hence embeddings) and final sinks (and hence quotients) in $\T\Mod$. We prove the main result of the paper in \S\ref{extensivity_section}. We begin \S\ref{extensivity_section} by stating the mild syntactic condition \eqref{variable_para} that we will impose on relational Horn theories to prove the infinite extensivity of their categories of models; this condition is satisfied by all of our examples, and many others. We then establish in \cref{extensive_thm} that if $\T$ is a relational Horn theory that satisfies the syntactic condition \eqref{variable_para}, then $\T\Mod$ is infinitely extensive. We thank the referee for useful comments and suggestions that improved the content and presentation of the paper.    

\section{Notation and background}
\label{background}

We first recall some background material on concrete and topological categories, which can be found (e.g.) in \cite{AHS}. 

\begin{para}
\label{concrete_cat}
{\em
A \emph{concrete category (over $\Set$)} is simply a category $\C$ equipped with a faithful functor $|-| : \C \to \Set$ (which we will usually not mention explicitly). Concrete categories $\left(\A, |-|_\A\right)$ and $\left(\B, |-|_\B\right)$ are \emph{concretely isomorphic} if there is an isomorphism of categories $F : \A \to \B$ satisfying $|-|_\B \circ F = |-|_\A$. A \emph{source} in an arbitrary category $\C$ is a (possibly large) class of morphisms $(h_i : X \to X_i)_{i \in I}$ in $\C$ with the same domain, while a \emph{sink} in $\C$ is a (possibly large) class of morphisms $(h_i : X_i \to X)_{i \in I}$ with the same codomain. 

A source $(h_i : X \to X_i)_{i \in I}$ in a concrete category $\C$ is \emph{initial} if for any $\C$-object $Y$ and function $h : |Y| \to |X|$, the function $h$ lifts to a $\C$-morphism $h : Y \to X$ iff the composite functions $h_i \circ h : |Y| \to |X_i|$ lift to $\C$-morphisms $h_i \circ h : Y \to X_i$ for all $i \in I$. In particular, a morphism of $\C$ is \emph{initial} if the source consisting of just that morphism is initial. Dually, a sink $(h_i : X_i \to X)_{i \in I}$ in a concrete category $\C$ is \emph{final} if for any $\C$-object $Y$ and function $h : |X| \to |Y|$, the function $h$ lifts to a $\C$-morphism $h : X \to Y$ iff the composite functions $h \circ h_i : |X_i| \to |Y|$ lift to $\C$-morphisms $h \circ h_i : X_i \to Y$ for all $i \in I$. In particular, a morphism of $\C$ is \emph{final} if the sink consisting of just that morphism is final.

If $\C$ is a concrete category, then a \emph{structured source} is a (possibly large) class of functions $\left(f_i : S \to |X_i|\right)_{i \in I}$ with $S$ a set and $X_i$ a $\C$-object for each $i \in I$, while a \emph{structured sink} is defined dually. A concrete category $\C$ is \emph{topological (over $\Set$)} if every structured source $\left(f_i : S \to |X_i|\right)_{i \in I}$ has an initial lift, meaning that there is a $\C$-object $X$ with $|X| = S$ such that each $f_i : |X| = S \to |X_i|$ lifts to a $\C$-morphism $f_i : X \to X_i$, and the resulting source $(f_i : X \to X_i)_{i \in I}$ is initial. A topological category over $\Set$ also satisfies the dual property that every structured sink has a final lift (see e.g. \cite[21.9]{AHS}).
}
\end{para}

\begin{para}
\label{fact_system}
{\em
Let $\C$ be a topological category over $\Set$. We say that a morphism of $\C$ is \emph{injective} (resp. \emph{surjective, bijective}) if its underlying function is injective (resp. surjective, bijective). Since $|-| : \C \to \Set$ is faithful and preserves small limits and colimits (see \cite[21.15]{AHS}), it follows that the monomorphisms (resp. epimorphisms) of $\C$ are precisely the injective (resp. surjective) morphisms. 

An \emph{embedding} of $\C$ is an injective morphism that is initial, while a \emph{quotient (morphism)} of $\C$ is a surjective morphism that is final. The embeddings (resp. quotients) are precisely the strong monomorphisms (resp. strong epimorphisms) of $\C$, which in turn are precisely the regular monomorphisms (resp. regular epimorphisms) of $\C$ (see e.g. \cite[21.13]{AHS}). It follows that the isomorphisms of $\C$ are precisely the bijective embeddings, or equivalently the bijective quotients.    
}
\end{para} 

\section{Relational Horn theories}
\label{first_section}

We begin by defining the notion of a \emph{relational signature} and its structures. 

\begin{defn}
\label{relational_sig}
{\em
A \textbf{relational signature} is a set $\Pi$ of \textbf{relation symbols} together with an assignment to each relation symbol of a finite arity, i.e.~a positive integer $n \geq 1$.
}
\end{defn}

\noindent We will typically write $R$ for a general relation symbol. We fix a relational signature $\Pi$ for the remainder of \S\ref{first_section}. The next two definitions are essentially taken from \cite[Definition 3.1]{Monadsrelational}.

\begin{defn}
\label{edge}
{\em
A \textbf{$\Pi$-edge} in a set $S$ is a pair $(R, (s_1, \ldots, s_n))$ consisting of a relation symbol $R \in \Pi$ (of arity $n \geq 1$) and an ordered $n$-tuple $(s_1, \ldots, s_n)$ of elements of $S$. A \textbf{$\Pi$-structure} $X$ consists of a set $|X|$ together with a subset $R^X \subseteq |X|^n$ for each relation symbol $R \in \Pi$ (of arity $n \geq 1$). We may equivalently describe a $\Pi$-structure $X$ as a set $|X|$ equipped with a set $\E(X)$ of $\Pi$-edges in $|X|$: if $R \in \Pi$ of arity $n \geq 1$, then $(x_1, \ldots, x_n) \in R^X$ iff $\E(X)$ contains the $\Pi$-edge $(R, (x_1, \ldots, x_n))$. We will pass between these equivalent descriptions of $\Pi$-structures without further comment. We will often write $X \models Rx_1\ldots x_n$ to mean that $(x_1, \ldots, x_n) \in R^X$.
}
\end{defn}   

\begin{defn}
\label{Pi_morphism}
{\em
Let $h : S \to S'$ be a function from a set $S$ to a set $S'$, and let $e = (R, (s_1, \ldots, s_n))$ be a $\Pi$-edge in $S$. We write $h \cdot e = h \cdot (R, (s_1, \ldots, s_n))$ for the $\Pi$-edge $(R, (h(s_1), \ldots, h(s_n)))$ in $S'$. If $E$ is a set of $\Pi$-edges in $S$, then we write $h \cdot E$ for the set of $\Pi$-edges $\{h \cdot e \mid e \in S\}$ in $S'$. If $E'$ is a set of $\Pi$-edges in $S'$, then we write $h^{-1}[E']$ for the set of $\Pi$-edges $e$ in $S$ such that $h \cdot e \in E'$.
 
Given $\Pi$-structures $X$ and $Y$, a \textbf{($\Pi$-)morphism $h : X \to Y$} is a function $h : |X| \to |Y|$ such that $h \cdot \E(X) \subseteq \E(Y)$, or equivalently such that $\E(X) \subseteq h^{-1}[\E(Y)]$. We let $\Str(\Pi)$ be the concrete category of $\Pi$-structures and $\Pi$-morphisms.    
}
\end{defn} 

\noindent We now describe the syntax of relational Horn theories.  

\begin{defn}
\label{Horn_formula}
{\em
Let $\lambda$ be a regular cardinal, and let $\Var$ be a set of variables of cardinality $\lambda$. A \textbf{$\lambda$-ary\footnote{When $\lambda = \aleph_0$, we will say ``finitary'' rather than ``$\aleph_0$-ary''.} Horn formula (over $\Pi$)} is an expression of the form $\Phi \Longrightarrow \psi$, where $\Phi$ is a set of $\Pi$-edges in $\Var$ of cardinality $< \lambda$ and $\psi$ is a $\left(\Pi \cup \{=\}\right)$-edge in $\Var$, where $=$ is a fresh binary relation symbol not in $\Pi$. If $\Phi = \{\varphi_1, \ldots, \varphi_n\}$ is finite, then we write $\varphi_1, \ldots, \varphi_n \Longrightarrow \psi$, and if $\Phi = \varnothing$, then we write $\Longrightarrow \psi$. A \textbf{$\lambda$-ary Horn formula without equality (over $\Pi$)} is a $\lambda$-ary Horn formula $\Phi \Longrightarrow \psi$ (over $\Pi$) such that $\psi$ is a $\Pi$-edge in $\Var$, i.e.~such that $\psi$ does not contain the fresh binary relation symbol $=$.    
}
\end{defn}

\begin{defn}
\label{Horn_theory}
{\em
Let $\lambda$ be a regular cardinal. A \textbf{$\lambda$-ary relational Horn theory $\T$ (without equality)} is a set of $\lambda$-ary Horn formulas (without equality) over $\Pi$, which we call the \emph{axioms} of $\T$. A \textbf{relational Horn theory (without equality)} is a $\lambda$-ary relational Horn theory (without equality) for some regular cardinal $\lambda$.   
}
\end{defn}

\begin{defn}
\label{formula_satis}
{\em
Let $\lambda$ be a regular cardinal, and let $X$ be a $\Pi$-structure. We let $\overline{X}$ be the $\left(\Pi \cup \{=\}\right)$-structure defined by $\left|\overline{X}\right| := |X|$ and $\E\left(\overline{X}\right) := \E(X) \cup \{(=, (x, x)) \mid x \in |X|\}$. A \textbf{valuation in $X$} is a function $\kappa : \Var \to |X|$. We say that $X$ \textbf{satisfies} a $\lambda$-ary Horn formula $\Phi \Longrightarrow \psi$ over $\Pi$ if whenever $\kappa$ is a valuation in $X$ such that $X \models \kappa \cdot \varphi$ for each $\varphi \in \Phi$, then $\Xbar \models \kappa \cdot \psi$. A $\Pi$-structure $X$ is a \textbf{model} of a $\lambda$-ary relational Horn theory $\T$ if $X$ satisfies every axiom of $\T$. We let $\T\Mod$ be the full subcategory of $\Str(\Pi)$ consisting of the models of $\T$, which is a concrete category when equipped with the faithful functor $|-| : \T\Mod \to \Set$ obtained by restricting the faithful functor $|-| : \Str(\Pi) \to \Set$.          
}
\end{defn}

\begin{egg}
\label{without_equality_examples}
{\em
We provide the following central examples of relational Horn theories. Some further examples may be found in \cite[Example 3.5]{Monadsrelational}.
\begin{enumerate}[leftmargin=*]
\item Let $\T$ be the empty relational Horn theory over $\Pi$. Then of course $\T\Mod = \Str(\Pi)$. In particular, if $\Pi$ is empty, then $\T\Mod = \Set$.\label{empty} 

\item Let $\Pi$ consist of a single binary relation symbol $\leq$, and let $\T$ be the finitary relational Horn theory without equality over $\Pi$ that consists of the two axioms $\Longrightarrow x \leq x$ and $x \leq y, y \leq z \Longrightarrow x \leq z$. Then $\T\Mod$ is the concrete category $\mathsf{Preord}$ of preordered sets and monotone maps. If one extends $\T$ by adding the further axiom $x \leq y, y \leq x \Longrightarrow x = y$, then the category of models of the resulting finitary relational Horn theory with equality is the concrete category $\mathsf{Pos}$ of posets and monotone maps.\label{preord}  

\item The following examples come from \cite[Definition 2.2 and Remark 2.4(2)]{Metagories}. Let $(\V, \leq, \tensor, \sfk)$ be a \emph{commutative unital quantale} \cite{Quantales}, meaning that $(\V, \leq)$ is a complete lattice and $(\V, \tensor, \sfk)$ is a commutative monoid such that $\tensor$ distributes over arbitrary suprema in each variable. A \emph{(small) $\V$-category} $(X, d)$ (see also \cite{Kelly}) is a set $X$ equipped with a function $d : X \times X \to \V$ satisfying the two conditions
\[ d(x, x) \geq \sfk \] 
\[ d(x, z) \geq d(x, y) \tensor d(y, z) \] for all $x, y, z \in X$. A \emph{pseudo-$\V$-metric space} is a $\V$-category $(X, d)$ such that $d : X \times X \to \V$ satisfies the further symmetry condition
\[ d(x, y) = d(y, x) \]
for all $x, y \in X$. Finally, a \emph{$\V$-metric space} is a pseudo-$\V$-metric space $(X, d)$ satisfying the further ``separation'' condition
\[ d(x, y) \geq \sfk \Longrightarrow x = y \] for all $x, y \in X$.

If $(X, d_X)$ and $(Y, d_Y)$ are $\V$-categories, then a \emph{$\V$-functor} or \emph{$\V$-contraction} $h : (X, d_X) \to (Y, d_Y)$ is a function $h : X \to Y$ satisfying $d_X(x, x') \leq d_Y(h(x), h(x'))$ for all $x, x' \in X$. We let $\V\Cat$ be the concrete category of $\V$-categories and $\V$-functors, we let $\PMet_\V$ be the full subcategory of $\V\Cat$ consisting of the pseudo-$\V$-metric spaces, and we let $\Met_\V$ be the full subcategory of $\PMet_\V$ consisting of the $\V$-metric spaces. We regard $\PMet_\V$ and $\Met_\V$ as concrete categories by suitably restricting the faithful functor $|-| : \V\Cat \to \Set$. As indicated in \cite[Example 2.3]{Metagories}, one has the following specific examples of $\V\Cat$, $\PMet_\V$, and $\Met_\V$ for suitable choices of $(\V, \leq, \tensor, \sfk)$ (see \cite[Example 2.1]{Metagories}):
\begin{itemize}[leftmargin=*]
\item For the trivial quantale $\mathds{1}$ with just one element, both $\mathds{1}\Cat$ and $\PMet_{\mathds{1}}$ are equivalent to $\Set$, while $\Met_{\mathds{1}}$ is equivalent to the terminal category. 
\item For the Boolean 2-chain quantale $\mathbf{2}$, we have that $\mathbf{2}\Cat$ is equivalent to the category $\mathsf{Preord}$ of preordered sets and monotone maps, while $\PMet_{\mathbf{2}}$ is equivalent to the category whose objects are sets equipped with an equivalence relation, and whose morphisms are functions preserving the equivalence relations. $\Met_{\mathbf{2}}$ is equivalent to $\Set$.  
\item For the \emph{Lawvere quantale} $\mathbb{R}_+ = ([0, \infty], \geq, +, 0)$ \cite{Lawveremetric} given by the extended real half line with the reverse ordering, we have that $\mathbb{R}_+\Cat$ is the category of \emph{Lawvere metric spaces}, while $\PMet_{\mathbb{R}_+}$ is the category $\PMet$ of extended pseudo-metric spaces (i.e.~pseudo-metric spaces where two points may have distance $\infty$) and contractions (i.e.~non-expanding maps), and $\Met_{\mathbb{R}_+}$ is the category $\Met$ of extended metric spaces (i.e.~extended pseudo-metric spaces $(X, d)$ satisfying $d(x, y) = 0 \Longrightarrow x = y$) and contractions. 
\item For the quantale $\Delta$ of \emph{distance distribution functions} (see e.g. \cite[\S 3.1]{Probmetenriched}), we have that $\Delta\Cat$ is equivalent to the category $\mathsf{ProbMet}$ of \emph{probabilistic metric spaces} (see \cite[\S 3.2]{Probmetenriched} and \cite{Probmetbook}), while $\PMet_\Delta$ is equivalent to the category of \emph{symmetric} probabilistic metric spaces, and $\Met_\Delta$ is equivalent to the category of \emph{symmetric, separated} probabilistic metric spaces.
\end{itemize}        

Let $\Pi_\V$ consist of just the binary relation symbols $\sim_v$ for each $v \in \V$. Let $\T_{\V\Cat}$ be the relational Horn theory without equality over $\Pi_\V$ that consists of the following axioms, where $v, v'$ range over $\V$:
\[ \Longrightarrow x \sim_\sfk x \]
\[ x \sim_v y, y \sim_{v'} z \Longrightarrow x \sim_{v \tensor v'} z \]
\[ x \sim_v y \Longrightarrow x \sim_{v'} y \tag*{($v \geq v'$)} \]
\[ \{x \sim_{v_i} y \mid i \in I\} \Longrightarrow x \sim_{\bigvee_i v_i} y \]
Then $\T_{\V\Cat}\Mod$ is concretely isomorphic \eqref{concrete_cat} to $\V\Cat$, which we prove in the Appendix (\S\ref{appendix}). Let $\T_{\PMet_\V}$ be the relational Horn theory without equality over $\Pi_\V$ that extends $\T_{\V\Cat}$ by adding the symmetry axioms
\[ x \sim_v y \Longrightarrow y \sim_v x \] for all $v \in \V$. Then $\T_{\PMet_\V}\Mod$ is concretely isomorphic to $\PMet_\V$, which we also prove in the Appendix (\S\ref{appendix}). Finally, let $\T_{\Met_\V}$ be the relational Horn theory with equality over $\Pi_\V$ that extends $\T_{\PMet_\V}$ by adding the axiom
\[ x \sim_\sfk y \Longrightarrow x = y. \] Then we also show in the Appendix (\S\ref{appendix}) that $\T_{\Met_\V}\Mod$ is concretely isomorphic to $\Met_\V$.\label{met}
\end{enumerate}
}
\end{egg}

\section{Topological properties of relational Horn theories without equality}
\label{second_section}

Throughout \S\ref{second_section} we fix a relational Horn theory $\T$ \emph{without equality} over a relational signature $\Pi$ (an assumption that we will occasionally repeat for emphasis). We first aim to characterize the initial sources and final sinks in the concrete category $\T\Mod$, which will allow us to prove that $\T\Mod$ is topological over $\Set$, which is in fact a special case of a result \cite[Proposition 5.1]{Rosickyconcrete} by Rosick\'{y}. We will then characterize the embeddings and quotients in $\T\Mod$. We first require the following definition. 

\begin{defn}
\label{deduc_relation}
{\em
Let $S$ be a set. A \textbf{$\T$-relation} on $S$ is a set $\calE$ of $\Pi$-edges in $S$ satisfying the following condition: for any axiom $\Phi \Longrightarrow \psi$ of $\T$ and any valuation $\kappa : \Var \to S$ such that $\kappa \cdot \varphi \in \calE$ for each $\varphi \in \Phi$, we have $\kappa \cdot \psi \in \calE$. The set of all $\Pi$-edges in $S$ is clearly a $\T$-relation on $S$, and the intersection of a small family of $\T$-relations on $S$ is a $\T$-relation on $S$. So for any set $\calE$ of $\Pi$-edges in $S$, we can define the \textbf{$\T$-closure} $\T(\calE)$ of $\calE$ to be the smallest $\T$-relation on $S$ that contains $\calE$, i.e.~the intersection of all $\T$-relations on $S$ that contain $\calE$. Note that a $\Pi$-structure $X$ is a $\T$-model iff $\E(X)$ is a $\T$-relation on $|X|$.
}
\end{defn}

\noindent The proof of the following lemma is elementary.

\begin{lem}
\label{inverse_Trel}
Let $h : S \to S'$ be a function from a set $S$ to a set $S'$, and let $\calE'$ be a $\T$-relation on $S'$. Then $h^{-1}\left[\calE'\right]$ is a $\T$-relation on $S$.
\end{lem}

\begin{prop}
\label{initial_final_sinks}
A source $(h_i : X \to X_i)_{i \in I}$ in $\T\Mod$ is initial iff for each $R \in \Pi$ of arity $n \geq 1$ and all $x_1, \ldots, x_n \in |X|$, we have $X \models Rx_1 \ldots x_n$ iff $X_i \models Rh_i(x_1)\ldots h_i(x_n)$ for all $i \in I$. A sink $(h_i : X_i \to X)_{i \in I}$ in $\T\Mod$ is final iff $\E(X)$ is the $\T$-closure of $\bigcup_{i \in I} h_i \cdot \E(X_i)$.
\end{prop}

\begin{proof}
For the first assertion, if $\E(X)$ has the stated characterization, then it readily follows that $(h_i)_i$ is initial. So assume that $(h_i)_i$ is initial; since each $h_i$ ($i \in I)$ is a $\Pi$-morphism, we need only prove that $X \models Rx_1\ldots x_n$ whenever $X_i \models Rh_i(x_1)\ldots h_i(x_n)$ for all $i \in I$. Assuming the latter condition, we define a $\Pi$-structure $Y$ by setting $|Y| := |X|$ and letting $\E(Y)$ be the $\T$-closure of $\E(X) \cup \{(R, (x_1, \ldots, x_n))\}$. Then $Y$ is a $\T$-model, and for each $i \in I$ we claim that the function $h_i : |Y| = |X| \to |X_i|$ is a $\Pi$-morphism $h_i : Y \to X_i$, i.e.~that $\E(Y) \subseteq h_i^{-1}[\E(X_i)]$. By definition of $\E(Y)$, it suffices to show that $h_i^{-1}[\E(X_i)]$ is a $\T$-relation on $|Y| = |X|$ that contains $\E(X) \cup \{(R, (x_1, \ldots, x_n))\}$. Since $\E(X_i)$ is a $\T$-relation on $|X_i|$ (because $X_i$ is a $\T$-model), we deduce from \cref{inverse_Trel} that $h_i^{-1}[\E(X_i)]$ is a $\T$-relation on $|Y| = |X|$. Then because $h_i : X \to X_i$ is a $\Pi$-morphism and $X_i \models Rh_i(x_1)\ldots h_i(x_n)$ by hypothesis, it follows that $h_i^{-1}[\E(X_i)]$ contains $\E(X) \cup \{(R, (x_1, \ldots, x_n))\}$. So each $h_i : Y \to X_i$ ($i \in I$) is a $\Pi$-morphism, and the initiality of the source $(h_i)_i$ then implies that the identity function $1_{|X|} : |Y| = |X| \to |X|$ is a $\Pi$-morphism $Y \to X$, which entails that $X \models Rx_1 \ldots x_n$, as desired.       

For the second assertion, suppose first that $\E(X)$ is the $\T$-closure of $\bigcup_{i \in I} h_i \cdot \E(X_i)$. To show that the sink $(h_i)_i$ is final, let $Y$ be a $\T$-model and let $h : |X| \to |Y|$ be a function such that the function $h \circ h_i : |X_i| \to |Y|$ is a $\Pi$-morphism $X_i \to Y$ for each $i \in I$. To show that $h$ is a $\Pi$-morphism $X \to Y$, we must show that $\E(X) \subseteq h^{-1}[\E(Y)]$. By assumption on $\E(X)$, it suffices to show that $h^{-1}[\E(Y)]$ is a $\T$-relation on $|X|$ that contains $\bigcup_{i \in I} h_i \cdot \E(X_i)$. The first claim holds by \cref{inverse_Trel} because $\E(Y)$ is a $\T$-relation on $|Y|$ (since $Y$ is a $\T$-model), and the second claim holds because each $h \circ h_i$ ($i \in I$) is a $\Pi$-morphism $X_i \to Y$. 

Now suppose that the sink $(h_i : X_i \to X)_{i \in I}$ is final, and let us show that $\E(X)$ must be the $\T$-closure of $\bigcup_{i \in I} h_i \cdot \E(X_i)$. That is, we must show that $\E(X)$ is the smallest $\T$-relation on $|X|$ that contains $\bigcup_{i \in I} h_i \cdot \E(X_i)$. That $\E(X)$ is a $\T$-relation holds because $X$ is a $\T$-model, and since each $h_i : X_i \to X$ ($i \in I$) is a $\Pi$-morphism, it follows that $\E(X)$ contains $\bigcup_{i \in I} h_i \cdot \E(X_i)$. Now let $\R$ be any $\T$-relation on $|X|$ that contains $\bigcup_{i \in I} h_i \cdot \E(X_i)$, and let us show that $\E(X) \subseteq \R$. Let $X'$ be the $\Pi$-structure defined by $|X'| := |X|$ and $\E(X') := \R$, so that $X'$ is a $\T$-model because $\R$ is a $\T$-relation. Showing that $\E(X) \subseteq \R$  is equivalent to showing that the identity function $1_{|X|} : |X| \to |X| = |X'|$ is a $\Pi$-morphism $X \to X'$. By finality of the sink $(h_i)_{i \in I}$, it then suffices to show that each function $h_i : |X_i| \to |X| = |X'|$ is a $\Pi$-morphism $X_i \to X'$, which is true by assumption on $\R$. This proves the desired characterization of $\E(X)$.           
\end{proof}

\noindent The following result, whose proof we outline in \cref{without_equality_sinks} below, is a special case of \cite[Prop.~5.1]{Rosickyconcrete}:

\begin{prop}[Rosick\'{y} \cite{Rosickyconcrete}]
\label{without_equality_top}
Let $\T$ be a relational Horn theory without equality. Then the concrete category $\T\Mod$ is topological over $\Set$. 
\end{prop}

\begin{para}
\label{without_equality_sinks}
{\em
The initial lift of a structured source $(h_i : S \to |X_i|)_{i \in I}$ is the source $(h_i : X \to X_i)_{i \in I}$, where $X$ is the $\Pi$-structure with $|X| := S$ and $X \models Rx_1\ldots x_n$ iff $X_i \models Rh_i(x_1) \ldots h_i(x_n)$ for all $i \in I$ (for any $R \in \Pi$ of arity $n \geq 1$ and any $x_1, \ldots, x_n \in |X|$). Since $\T$ is a relational Horn theory \emph{without equality} and each $X_i$ ($i \in I)$ is a $\T$-model, it readily follows that $X$ is a $\T$-model, and the source $(h_i)_i$ in $\T\Mod$ is initial by \cref{initial_final_sinks}. The final lift of a structured sink $(h_i : |X_i| \to S)_{i \in I}$ is the sink $(h_i : X_i \to X)_{i \in I}$, where $X$ is the $\Pi$-structure with $|X| := S$ and $\E(X)$ the $\T$-closure of $\bigcup_{i \in I} h_i \cdot \E(X_i)$. Then $X$ is a $\T$-model because $\E(X)$ is a $\T$-relation, and the sink $(h_i)_i$ is final by \cref{initial_final_sinks}.
}
\end{para} 

\begin{para}
\label{limits}
{\em
Given a small diagram $D : \B \to \T\Mod$, the limit cone of $D$ is the initial lift of the limit cone of $|-| \circ D$ in $\Set$, while the colimit cocone of $D$ is the final lift of the colimit cocone of $|-| \circ D$ in $\Set$ (see e.g. \cite[21.15]{AHS}). In particular, the functor $|-| : \T\Mod \to \Set$ strictly preserves small limits and colimits. If $\T$ is a relational Horn theory \emph{with} equality, then $\T\Mod$ is complete and cocomplete (e.g.~because $\T\Mod$ is locally presentable, by \cite[Proposition 5.30]{LPAC}), and it is still true that $|-| : \T\Mod \to \Set$ strictly preserves limits, since the inclusion $\T\Mod \hookrightarrow \Str(\Pi)$ preserves limits.    
}
\end{para}

\begin{defn}
\label{reflects_relations}
{\em
Let $h : X \to Y$ be a morphism in $\Str(\Pi)$. Then $h$ \textbf{reflects relations} if for each $R \in \Pi$ of arity $n \geq 1$ and any $x_1, \ldots, x_n \in |X|$, we have $X \models Rx_1 \ldots x_n$ if $Y \models Rh(x_1)\ldots h(x_n)$. 
}
\end{defn}

\noindent The characterizations of initial sources and final sinks in $\T\Mod$ provided in \cref{initial_final_sinks} immediately entail that embeddings and quotients in $\T\Mod$ have the following characterizations:

\begin{prop}
\label{embedding_prop}
Let $h : X \to Y$ be a morphism in $\T\Mod$. Then $h$ is an embedding iff $h$ is injective and reflects relations, while $h$ is a quotient morphism iff $h$ is surjective and $\E(Y)$ is the $\T$-closure of $h \cdot \E(X)$.
\end{prop}

If $\C$ is a topological category over $\Set$, then a bimorphism in $\C$ (i.e.~a morphism that is both epic and monic) is precisely a bijective morphism in view of \cref{fact_system}. As in \cite[Definition 16.1]{AHS}, we say that a full replete\footnote{Recall that a full subcategory $\B \hookrightarrow \C$ is \emph{replete} if whenever $C$ is an object of $\C$ that is isomorphic to an object of $\B$, then $C$ is an object of $\B$.} subcategory $\B \hookrightarrow \C$ is \emph{bireflective} if every object of $\C$ has a $\B$-reflection morphism that is a bimorphism of $\C$, i.e.~that is bijective. Since the concrete category $\Str(\Pi)$ is topological over $\Set$ in view of \cref{without_equality_examples}.\ref{empty} and \cref{without_equality_top}, we now have the following result (cf. \cite[Proposition 3.6]{Monadsrelational}): 

\begin{prop}
\label{without_equality_reflective}
Let $\T$ be a relational Horn theory without equality. Then the full replete subcategory $\T\Mod \hookrightarrow \Str(\Pi)$ is bireflective.
\end{prop}

\begin{proof}
Let $X$ be a $\Pi$-structure. We define a $\Pi$-structure $X^*$ by setting $|X^*| := |X|$ and letting $\E(X^*)$ be the $\T$-closure of $\E(X)$. Then $X^*$ is a $\T$-model, and the identity function $1_{|X|} : |X| \to |X| = |X^*|$ is a bijective $\Pi$-morphism, which we claim is a $\T\Mod$-reflection morphism for $X$. So let $h : X \to Y$ be a $\Pi$-morphism from $X$ to a $\T$-model $Y$. Then the function $h : |X^*| = |X| \to |Y|$ is also a $\Pi$-morphism $h : X^* \to Y$, because $\E(X^*)$ is the smallest $\T$-relation on $|X|$ that contains $\E(X)$, and thus $\E(X^*) \subseteq h^{-1}[\E(Y)]$ because $h^{-1}[\E(Y)]$ is a $\T$-relation on $|X|$ \eqref{inverse_Trel} that contains $\E(X)$ (since $h : X \to Y$ is a $\Pi$-morphism).      
\end{proof}

\section{Extensivity of $\T\Mod$ for a relational Horn theory $\T$}
\label{extensivity_section}

Throughout \S\ref{extensivity_section}, we fix a relational Horn theory $\T$ over a relational signature $\Pi$. 

\begin{defn}
\label{variable_set}
{\em
Let $\Var$ be a set of variables. For any $(\Pi \cup \{=\})$-edge $\varphi$ in $\Var$, we define the set of variables $\Var(\varphi)$ occurring in $\varphi$ as follows: if $\varphi = (R, (v_1, \ldots, v_n))$ for some $R \in \Pi \cup \{=\}$ of arity $n \geq 1$ and some $v_1, \ldots, v_n \in \Var$, then $\Var(\varphi) := \{v_1, \ldots, v_n\}$. If $\Phi$ is a set of $(\Pi \cup \{=\})$-edges in $\Var$, then we define $\Var(\Phi) := \bigcup_{\varphi \in \Phi} \Var(\varphi)$.
}
\end{defn}

\begin{condition}
\label{variable_para}
{\em
We suppose throughout \S\ref{extensivity_section} that $\T$ satisfies the following mild syntactic condition. For each axiom $\Phi \Longrightarrow \psi$ of $\T$, we require that:
\begin{enumerate}[leftmargin=*]
\item Any two distinct elements of $\Phi$ share at least one variable in common; i.e.~if $\varphi, \varphi' \in \Phi$ are distinct, then $\Var(\varphi) \cap \Var(\varphi') \neq \varnothing$.\label{variable_1}
\item If $\Phi \neq \varnothing$, then $\Var(\psi) \subseteq \Var(\Phi)$; and if $\Phi = \varnothing$, then $\Var(\psi)$ is a singleton.\label{variable_2} 
\end{enumerate}
This condition is certainly satisfied by all of our central examples \eqref{without_equality_examples} (and by the additional examples of \cite[Example 3.5]{Monadsrelational}), and will be used to conveniently characterize small coproducts in $\T\Mod$ in \cref{without_equality_coprod} below. We shall provide some commentary on this condition in \cref{not_nec} and \cref{finally_dense} below.
} 
\end{condition}  

\begin{para}
\label{without_equality_pb}
{\em
In view of \cref{limits}, the pullback $A \times_{f, g} B$ in $\T\Mod$ of two $\T$-model morphisms $f : A \to C$ and $g : B \to C$ is formed by taking the initial lift of the pullback of the underlying functions $f : |A| \to |C|$ and $g : |B| \to |C|$ in $\Set$. So we have \[ \left|A \times_{f, g} B\right| = |A| \times_{f, g} |B| = \{(a, b) \in |A| \times |B| \mid f(a) = g(b)\} \] with
\[ R^{A \times_{f, g} B} = \left\{((a_1, b_1), \ldots, (a_n, b_n)) \in \left|A \times_{f, g} B\right|^n \mid A \models R a_1 \ldots a_n \text{ and } B \models R b_1 \ldots b_n\right\} \] for each $R \in \Pi$ of arity $n \geq 1$. The pullback projections $\pi_A : |A| \times_{f, g} |B| \to |A|$ and $\pi_B : |A| \times_{f, g} |B| \to |B|$ in $\Set$ then lift to pullback projections in $\T\Mod$.
}  
\end{para}

\begin{para}
\label{without_equality_coprod}
{\em
Let $\T$ be a relational Horn theory \emph{without equality}. In view of \cref{limits}, the coproduct $\coprod_{i \in I} X_i$ of a small family of $\T$-models $X_i$ ($i \in I$) is formed by taking the final lift of the coproduct of the underlying sets $|X_i|$ ($i \in I$) in $\Set$. So we have $\left|\coprod_i X_i\right| = \coprod_i |X_i|$, and $\E\left(\coprod_i X_i\right)$ is the $\T$-closure of the set of all $\Pi$-edges in the images of all coproduct injections $s_i : |X_i| \to \coprod_i |X_i|$ ($i \in I$). We now claim that in fact
\[ \E\left(\coprod_i X_i\right) = \bigcup_{i \in I} s_i \cdot \E(X_i), \] for which it suffices to show that $\bigcup_{i \in I} s_i \cdot \E(X_i)$ is a $\T$-relation on $\coprod_i |X_i|$. 

So let $\Phi \Longrightarrow \psi$ be an axiom of $\T$, let $\kappa : \Var \to \coprod_i |X_i|$ be a valuation, and suppose that $\kappa \cdot \varphi \in \bigcup_{i \in I} s_i \cdot \E(X_i)$ for each $\varphi \in \Phi$; we must show that $\kappa \cdot \psi \in \bigcup_{i \in I} s_i \cdot \E(X_i)$. Since any two distinct elements of $\Phi$ share at least one variable in common (see \cref{variable_para}.\ref{variable_1}) and the union $\bigcup_{i \in I} s_i \cdot \E(X_i)$ is disjoint (since the coproduct injections $s_i$ ($i \in I$) have disjoint images), there must be some $i \in I$ such that $\kappa \cdot \varphi \in s_i \cdot \E(X_i)$ for all $\varphi \in \Phi$. In view of \cref{variable_para}.\ref{variable_2}, we may then assume without loss of generality that $\kappa$ factors through $s_i : |X_i| \to \coprod_i |X_i|$ via a valuation $\kappa' : \Var \to |X_i|$, so that $\kappa = s_i \circ \kappa'$ and hence $s_i \cdot \kappa' \cdot \varphi \in s_i \cdot \E(X_i)$ for each $\varphi \in \Phi$. Since $s_i$ is injective, it follows that $\kappa' \cdot \varphi \in \E(X_i)$ for each $\varphi \in \Phi$, so that $\kappa' \cdot \psi \in \E(X_i)$ because $X_i$ is a $\T$-model. It then follows that $\kappa \cdot \psi \in \bigcup_{i \in I} s_i \cdot \E(X_i)$, as desired. 

Now let $\T$ be an arbitrary relational Horn theory (possibly with equality), and let $\T^-$ be the relational Horn theory \emph{without equality} obtained from $\T$ by removing all axioms $\Phi \Longrightarrow \psi$ of $\T$ that contain equality, i.e.~where $\psi$ is a $\{=\}$-edge in $\Var$. We claim that the inclusion $\T\Mod \hookrightarrow \T^-\Mod$ preserves small coproducts. So let $(X_i)_{i \in I}$ be a small family of $\T$-models, and let us show that the coproduct $\coprod_i X_i$ in $\T^-\Mod$ is a $\T$-model. Since $\coprod_i X_i$ is a $\T^-$-model, it just remains to show that $\coprod_i X_i$ satisfies each axiom $\Phi \Longrightarrow \psi$ of $\T$ where $\psi$ is a $\{=\}$-edge $x = y$. So let $\kappa : \Var \to \coprod_i |X_i|$ be a valuation such that $\kappa \cdot \varphi \in \E\left(\coprod_i X_i\right) = \bigcup_i s_i \cdot \E(X_i)$ for each $\varphi \in \Phi$, and let us show that $\kappa(x) = \kappa(y)$. Since $\T$ satisfies \cref{variable_para}, we deduce exactly as in the previous paragraph that there must be some $i \in I$ such that $\kappa$ factors through a valuation $\kappa' : \Var \to |X_i|$ (so that $s_i \circ \kappa' = \kappa$) and $\kappa' \cdot \varphi \in \E(X_i)$ for all $\varphi \in \Phi$. Because $X_i$ is a $\T$-model, we then deduce that $\kappa'(x) = \kappa'(y)$, so that $\kappa(x) = \kappa(y)$ as desired. Since $|-| : \T^-\Mod \to \Set$ preserves small coproducts, it follows that $|-| : \T\Mod \to \Set$ preserves small coproducts. 

Let $\T$ be a relational Horn theory and let $(X_i)_{i \in I}$ be a small family of $\T$-models. For each $R \in \Pi$ of arity $n \geq 1$ and any $(i_1, x_1), \ldots, (i_n, x_n) \in \coprod_i |X_i|$, we thus have $\coprod_i X_i \models R(i_1, x_1)\ldots(i_n, x_n)$ iff $i_1 = \ldots = i_n = i$ and $X_i \models Rx_1\ldots x_n$. Therefore, each coproduct injection $s_i : X_i \to \coprod_i X_i$ ($i \in I$) is an embedding\footnote{Even if $\T$ is a relational Horn theory \emph{with} equality, a morphism of $\T\Mod$ that is injective and reflects relations is still an embedding.} \eqref{embedding_prop}. The initial object of $\T\Mod$ is the empty set equipped (of course) with the empty set of $\Pi$-edges. 
}
\end{para}

We recall from \cite{Carboniextensive} that a category $\C$ is said to be \emph{infinitely extensive} if it has small coproducts and for any small family $(X_i)_{i \in I}$ of objects of $\C$, the canonical functor $\prod_{i \in I} \C/X_i \to \C/\left(\coprod_i X_i\right)$ is an equivalence. If $\C$ has small coproducts and pullbacks, recall that small coproducts in $\C$ are said to be \emph{universal} (or \emph{stable under pullback}) if for any small family $(X_i)_{i \in I}$ of objects of $\C$ with coproduct $\left(s_i : X_i \to \coprod_i X_i\right)_{i \in I}$ and any morphism $f : Y \to \coprod_i X_i$, if the following diagram is a pullback in $\C$ for each $i \in I$: 
\[\begin{tikzcd}
	{P_i} && Y \\
	\\
	{X_i} && {\coprod_i X_i},
	\arrow["{t_i}", from=1-1, to=1-3]
	\arrow["f", from=1-3, to=3-3]
	\arrow["{\pi_i}"', from=1-1, to=3-1]
	\arrow["{s_i}"', from=3-1, to=3-3]
\end{tikzcd}\]
then $(t_i : P_i \to Y)_{i \in I}$ is a coproduct diagram in $\C$. If $\C$ has small coproducts and pullbacks, then a coproduct $\left(s_i : X_i \to \coprod_i X_i\right)_{i \in I}$ of a small family $(X_i)_{i \in I}$ of objects of $\C$ is said to be \emph{disjoint} if for any distinct indices $i, j \in I$, the pullback of $s_i$ along $s_j$ is isomorphic to the initial object of $\C$. If $\C$ has small coproducts and pullbacks, then by (the infinitary version of) \cite[Proposition 2.14]{Carboniextensive}, we have that $\C$ is infinitely extensive iff small coproducts in $\C$ are universal and disjoint. 

\begin{theo}
\label{extensive_thm}
Let $\T$ be a relational Horn theory that satisfies \cref{variable_para}. Then $\T\Mod$ is infinitely extensive.
\end{theo}

\begin{proof}
Since $\T\Mod$ has small coproducts and pullbacks \eqref{limits}, it is equivalent to show that small coproducts in $\T\Mod$ are universal and disjoint. For the first claim, let $(X_i)_{i \in I}$ be a small family of $\T$-models with coproduct $(s_i : X_i \to \coprod_i X_i)_{i \in I}$ in $\T\Mod$. Let $f : Y \to \coprod_i X_i$ be a morphism of $\T\Mod$, and for each $i \in I$ suppose that the following diagram is a pullback in $\T\Mod$:
\[\begin{tikzcd}
	{P_i} && Y \\
	\\
	{X_i} && {\coprod_i X_i}.
	\arrow["{t_i}", from=1-1, to=1-3]
	\arrow["f", from=1-3, to=3-3]
	\arrow["{\pi_i}"', from=1-1, to=3-1]
	\arrow["{s_i}"', from=3-1, to=3-3]
\end{tikzcd}\] 
We must show that $(t_i : P_i \to Y)_{i \in I}$ is a coproduct diagram in $\T\Mod$. Because $\Set$ is infinitely extensive and $|-| : \T\Mod \to \Set$ preserves small coproducts and pullbacks (see \cref{limits} and \cref{without_equality_coprod}), we deduce that $(t_i : |P_i| \to |Y|)_{i \in I}$ is a coproduct diagram in $\Set$. Now let $(h_i : P_i \to Z)_{i \in I}$ be a small family of $\T$-model morphisms. Then there is a unique function $h : |Y| \to |Z|$ satisfying $h \circ t_i = h_i$ for each $i \in I$, so we just have to show that $h$ is a $\Pi$-morphism $Y \to Z$. So let $R \in \Pi$ of arity $n \geq 1$ and suppose that $Y \models R y_1 \ldots y_n$. Since $f : Y \to \coprod_i X_i$ is a $\Pi$-morphism, we obtain $\coprod_i X_i \models Rf(y_1) \ldots f(y_n)$. Then by \cref{without_equality_coprod}, we deduce that there are some $i \in I$ and some $x_1, \ldots, x_n \in |X_i|$ such that $s_i(x_k) = f(y_k)$ for each $1 \leq k \leq n$ and $X_i \models Rx_1 \ldots x_n$. For each $1 \leq k \leq n$ we then have $(x_k, y_k) \in |P_i|$, and we have $P_i \models R(x_1, y_1)\ldots(x_n, y_n)$ because $X_i \models Rx_1 \ldots x_n$ and $Y \models R y_1 \ldots y_n$ (see ~\ref{without_equality_pb}). Since $h_i : P_i \to Z$ is a $\Pi$-morphism, we then obtain $Z \models R h_i(x_1, y_1)\ldots h_i(x_n, y_n)$, i.e.~$Z \models R h(t_i(x_1, y_1))\ldots h(t_i(x_n, y_n))$, i.e.~$Z \models R h(y_1)\ldots h(y_n)$, as desired. This proves that small coproducts are universal in $\T\Mod$.

It remains to show that the coproduct $(s_i : X_i \to \coprod_i X_i)_{i \in I}$ is disjoint. So let $i, j \in I$ be distinct. Since small coproducts are disjoint in $\Set$ and $|-| : \T\Mod \to \Set$ preserves pullbacks and small coproducts, we deduce that the underlying set of the pullback of $s_i$ along $s_j$ in $\T\Mod$ is $\varnothing$. But there is a unique $\T$-model with underlying set $\varnothing$, and this is the initial object of $\T\Mod$ \eqref{without_equality_coprod}.             
\end{proof}

\noindent Recall from \cite{Carboniextensive} that a category $\C$ is said to be \emph{infinitely distributive} if it has finite products and small coproducts and for each $X \in \mathop{\mathsf{ob}}\C$, the functor $X \times (-) : \C \to \C$ preserves small coproducts. The following result now follows immediately from \cref{extensive_thm} and (the infinitary version of) \cite[Proposition 4.5]{Carboniextensive}:

\begin{cor}
\label{distributive_cor}
Let $\T$ be a relational Horn theory that satisfies \cref{variable_para}. Then $\T\Mod$ is infinitely distributive.
\end{cor}

\begin{egg}
\label{not_nec}
{\em
Let $\T$ be a relational Horn theory. \cref{extensive_thm} shows that the satisfaction of \cref{variable_para} is sufficient for infinite extensivity of $\T\Mod$, but it is not \emph{necessary}, as the following (rather trivial) examples show. 
\begin{enumerate}[leftmargin=*]
\item Let $\Pi$ be any relational signature that contains at least one relation symbol of arity $\geq 2$, and suppose that $\T$ consists of just the axioms $\Longrightarrow Rv_1\ldots v_n$ for each $R \in \Pi$ of arity $n \geq 1$, where $v_1, \ldots, v_n$ are pairwise distinct variables. Then $\T$ clearly violates \cref{variable_para}.\ref{variable_2}. A $\Pi$-structure $X$ is a $\T$-model iff $R^X = |X|^n$ for each $R \in \Pi$ of arity $n \geq 1$, and hence there is a unique $\T$-model structure on any set (the \emph{indiscrete} $\Pi$-structure), and moreover $\T\Mod(X, Y) = \Set(|X|, |Y|)$ for any $\T$-models $X, Y$. It follows that the functor $|-| : \T\Mod \to \Set$ is an isomorphism, so that $\T\Mod$ is infinitely extensive because $\Set$ is. 

\item As another (trivial) example, let $\Pi$ be the empty relational signature, and suppose that $\T$ consists of just the axiom $\Longrightarrow x = y$ for distinct variables $x, y$. Then $\T$ again violates \cref{variable_para}.\ref{variable_2}, but clearly $\T\Mod$ is equivalent to the terminal category, which is (trivially) infinitely extensive. 
\end{enumerate}
We are not aware of any ``natural'' or well-studied examples of relational Horn theories that fail to satisfy \cref{variable_para}, and we also do not know whether there is an alternative syntactic condition that is both sufficient \emph{and} necessary for infinite extensivity of $\T\Mod$. 
}
\end{egg}

\begin{rmk}
\label{finally_dense}
{\em
Let $\T$ be a relational Horn theory \emph{without equality}, and let $R \in \Pi$ be a relation symbol of arity $n \geq 1$. We define a $\T$-model $R_\T$ by setting $\left|R_\T\right| := \{1, \ldots, n\}$ and letting $\E\left(R_\T\right)$ be the $\T$-closure of the set consisting of the single $\Pi$-edge $(R, (1, \ldots, n))$ on $\{1, \ldots, n\}$; in other words, $R_\T$ is the free $\T$-model on the $\Pi$-structure with underlying set $\{1, \ldots, n\}$ and the unique $\Pi$-edge $(R, (1, \ldots, n))$ (see the proof of \cref{without_equality_reflective}). For any $\T$-model $X$, it follows that $\Pi$-morphisms $R_\T \to X$ are in bijective correspondence with $\Pi$-edges in $\E(X)$ whose first component is $R$. The small full subcategory $\Pi_\T$ of $\T\Mod$ consisting of the $\T$-models $R_\T$ ($R \in \Pi$) is \emph{finally dense} in $\T\Mod$ (see \cite[Definition 10.69]{AHS}), which means that for any $\T$-model $X$, there is a final sink $(h_i : X_i \to X)_{i \in I}$ with codomain $X$ and $X_i \in \Pi_\T$ for each $i \in I$. Specifically, one takes $I := \E(X)$, and for $e = (R, (x_1, \ldots, x_n)) \in \E(X)$, one takes $X_e := R_\T$ and $h_e : R_\T \to X$ to be the $\Pi$-morphism corresponding to the $\Pi$-edge $e$. Then the resulting sink $(h_e : X_e \to X)_{e \in \E(X)}$ is final by \cref{initial_final_sinks}, because $\E(X)$ is clearly the $\T$-closure of $\bigcup_{e \in \E(X)} h_e \cdot \E(X_e)$.

Now if $\T$ satisfies \cref{variable_para}, then by \cref{without_equality_coprod} it readily follows that for each $R \in \Pi$, the $\T$-model $R_\T$ is a \emph{connected} object of $\T\Mod$, meaning that the representable functor $\T\Mod(R_\T, -) : \T\Mod \to \Set$ preserves small coproducts. We conclude that if $\T$ is a relational Horn theory without equality that satisfies \cref{variable_para}, then $\T\Mod$ has a small finally dense subcategory of connected objects. Given that $\T\Mod$ is topological over $\Set$ \eqref{without_equality_top} and infinitely extensive \eqref{extensive_thm}, one may now wonder about the following conjecture: 
\begin{conjecture}
Let $\C$ be a topological category over $\Set$ that has a finally dense subcategory of connected objects. Then $\C$ is infinitely extensive.
\end{conjecture}
\noindent We thank Rory Lucyshyn-Wright for discussions that led to the posing of this conjecture.    
}
\end{rmk}

\section{Appendix}
\label{appendix}

In this Appendix, we prove the claims asserted at the end of \cref{without_equality_examples}.\ref{met}. So let \linebreak $(\V, \leq, \tensor, \sfk)$ be a commutative unital quantale, and let $\T := \T_{\V\Cat}$. We first show that $\T\Mod$ is concretely isomorphic to $\V\Cat$. We first define a concrete functor $F : \T\Mod \to \V\Cat$, i.e.~a functor that commutes with the faithful functors to $\Set$. So let $X$ be a $\T$-model. We define a $\V$-category $FX = (|X|, d_X)$ by setting $d_X(x, y) := \bigvee \{v \in \V \mid X \models x \sim_v y\}$ for any $x, y \in |X|$. For each $x \in |X|$ we have $d_X(x, x) \geq \sfk$ because $X \models x \sim_\sfk x$. Now let $x, y, z \in |X|$, and let us show that $d(x, z) \geq d(x, y) \tensor d(y, z)$, i.e.~that
\begin{align*} 
\bigvee\{v \in \V \mid X \models x \sim_v z\}	&\geq \bigvee\{v' \in \V \mid X \models x \sim_{v'} y\} \tensor \bigvee\{v'' \in \V \mid X \models y \sim_{v''} z\} \\ 
&= \bigvee\{v' \tensor v'' \mid v', v'' \in \V, X \models x \sim_{v'} y \text{ and } X \models y \sim_{v''} z\},
\end{align*}
where the equality holds because $\tensor$ preserves arbitrary suprema in each variable separately. For any $v', v'' \in \V$ such that $X \models x \sim_{v'} y$ and $X \models y \sim_{v''} z$, we have $X \models x \sim_{v' \tensor v''} z$ and hence $\bigvee\{v \in \V \mid X \models x \sim_v z\} \geq v' \tensor v''$, which yields the desired inequality. This proves that $FX = (|X|, d_X)$ is a well-defined $\V$-category. If $h : X \to Y$ is a morphism of $\T$-models, then the function $h : |X| \to |Y|$ is a $\V$-functor $h : (|X|, d_X) \to (|Y|, d_Y)$ because for any $x, y \in |X|$ we have \[ d_Y(h(x), h(y)) = \bigvee\{v \in \V \mid Y \models h(x) \sim_v h(y)\} \geq \bigvee\{v \in \V \mid X \models x \sim_v y\} = d_X(x, y), \] because $X \models x \sim_v y$ implies $Y \models h(x) \sim_v h(y)$. So we set $F(h) := h$, and $F : \T\Mod \to \V\Cat$ is then clearly functorial and commutes with the faithful functors to $\Set$.

We now define a functor $G : \V\Cat \to \T\Mod$. So let $(X, d)$ be a $\V$-category. We define a $\Pi_\V$-structure $G(X, d)$ by setting $\left|G(X, d)\right| := X$ and, for any $v \in \V$ and $x, y \in X$, by setting $G(X, d) \models x \sim_v y$ iff $d(x, y) \geq v$. It is essentially immediate that $G(X, d)$ is a $\T$-model. Now let $h : (X, d_X) \to (Y, d_Y)$ be a $\V$-functor. Then the function $h : X \to Y$ is a $\Pi_\V$-morphism $h : G(X, d_X) \to G(Y, d_Y)$, because for any $v \in \V$ and $x, y \in X$ we have the implications
\[ G(X, d_X) \models x \sim_v y \ \Longleftrightarrow \ d_X(x, y) \geq v \ \Longrightarrow \ d_Y(h(x), h(y)) \geq v \ \Longleftrightarrow \ G(Y, d_Y) \models h(x) \sim_v h(y). \] So we set $G(h) : = h$, and $G : \V\Cat \to \T\Mod$ is then clearly functorial. 

We now show that $F$ and $G$ are mutually inverse on objects, which will complete the proof. First let $X$ be a $\T$-model, and let us show that $X = GFX = G(|X|, d_X)$, i.e.~that for any $v \in \V$ and $x, y \in |X|$ we have $X \models x \sim_v y$ iff $d_X(x, y) = \bigvee\{v' \in \V \mid X \models x \sim_{v'} y\} \geq v$. The forward implication is immediate. Now suppose that $w := \bigvee\{v' \in \V \mid X \models x \sim_{v'} y\} \geq v$, and let us show that $X \models x \sim_v y$. Since $X$ is a $\T$-model, we have $X \models x \sim_w y$. Then since $w \geq v$ (and $X$ is a $\T$-model), we obtain $X \models x \sim_v y$, as desired. 

Now let $(X, d)$ be a $\V$-category, and let us show that $(X, d) = FG(X, d) = \left(X, d_{G(X, d)}\right)$. So for all $x, y \in X$, we must show that $d(x, y) = d_{G(X, d)}(x, y)$, i.e.~that \[ d(x, y) = \bigvee\{v \in \V \mid G(X, d) \models x \sim_v y\} = \bigvee\{v \in \V \mid d(x, y) \geq v\}, \] which is immediate. This completes the proof that $\T\Mod = \T_{\V\Cat}\Mod$ is concretely isomorphic to $\V\Cat$. It is clear that if $X$ is a model of $\T_{\V\Cat}$, then $X$ is a model of $\T_{\PMet_\V}$ iff the associated $\V$-category $FX$ is a pseudo-$\V$-metric space, whence we obtain the further concrete isomorphism $\T_{\PMet_\V}\Mod \cong \PMet_\V$. Finally, it is easy to verify that if $X$ is a model of $\T_{\PMet_\V}$, then $X$ is a model of $\T_{\Met_\V}$ iff the associated pseudo-$\V$-metric space $FX$ is a $\V$-metric space, whence we obtain the concrete isomorphism $\T_{\Met_\V}\Mod \cong \Met_\V$.

\bibliographystyle{amsplain}
\bibliography{mybib}

\end{document}